\documentclass{article}
\usepackage{amssymb,amsthm,verbatim}
\newtheorem{theorem}{Theorem}
\newtheorem{lemma}{Lemma}
\newtheorem{proposition}{Proposition}
\newtheorem{corollary}{Corollary}
\newtheorem{question}{Question}
\begin{document}

\title{On the maximum rank of a real binary form}
\author{A.~Causa  \and R.~Re}
\date{}
\maketitle
\begin{abstract} We show that a real binary form $f$ of degree
$n\geq 3$ has $n$ distinct real roots if and only if for any
$(\alpha,\beta)\in\mathbb{R}^2\setminus\{0\}$ all the forms
$\alpha f_x+\beta f_y$ have $n-1$ distinct real roots. This
answers to a question of P. Comon and G. Ottaviani in \cite{coot},
and allows to complete their argument to show that $f$ has
symmetric rank $n$ if and only if it has $n$ distinct real roots.
\end{abstract}
\section{Introduction}
This paper deals with the following problem: Given a degree $n$ polynomial $f\in\mathbb
K[x_1,\ldots,x_m]$ find the rank (or Waring rank) of $f$ i.e.~the minimum number of summands
which achieve the following decomposition:\[f=\lambda_1l_1^n+\cdots
+\lambda_rl_r^n\quad\mbox{with }\lambda_i\in\mathbb K\mbox{ and }l_i\mbox{ linear forms.}\]
 For $\mathbb K=\mathbb C$ and $f$ generic the answer has been given
 (see \cite{alhi,otbr}), nevertheless some questions remain unsolved, e.g.~it's not yet
 known which is the stratification of the set of complex polynomials by the rank. However one
 can see \cite{cose} for an answer in the binary case.\\
 In the real case, i.e.~$\mathbb K=\mathbb R$, the situation becomes more complicated. In
 contrast to the complex case which has a generic rank, in the real case the generic rank is
 substituted by the concept of typical rank.
A rank $k$ is said \emph{typical} for a given degree $n$ if there exists a euclidean open set in
the space of real degree $n$ polynomials such that any $f$ in such open et has rank $k$.
 We will prove the following theorem, posed as a question in \cite{coot}.
\begin{theorem}\label{thm:main} Let $f(x,y)$ be a real homogeneous polynomial of degree $n\geq 3$ without multiple roots in $\mathbb{C}$.
 Then $f$ has all real roots if and only if for any $(0,0)\not=(\alpha,\beta)\in\mathbb{R}^2$ the polynomial $\alpha f_x+\beta f_y$ has
 $n-1$ distinct real roots.\end{theorem}
Notice that the ``only if" part of the theorem is easy. Indeed given
any $(\alpha,\beta)\not=(0,0)$ one may consider a new coordinate
system $l,$ $m$  on the projective line, such that $x=\alpha
l+\alpha' m$, and $y=\beta l+\beta' m$, so that
$\partial_l=\alpha\partial_x+\beta\partial_y$. Writing $f$ as a
function of $l,m$ and de-homogenizing by setting $m=1$ one sees
that $f_l$ has $n-1$ distinct roots by the theorem of Rolle.

In \cite{coot} the result above has been considered in connection
with the problem of determining the \emph{rank} of a real binary
form, that is the minimum number $r$ such that $f(x,y)=\lambda_1
l_1^n+\cdots+\lambda_r l_r^n$, with $\lambda_i\in\mathbb{R}$ and
$l_i=\alpha_ix+\beta_i y\in \mathbb{R}[x,y]$ for $i=1,\ldots r$.
Using the arguments already given in \cite{coot} and applying
Theorem \ref{thm:main}, one gets the following result.
\begin{corollary}\label{cor:rank} A real binary form $f(x,y)$ of degree $n\geq 3$ without multiple roots in $\mathbb{C}$ has rank $n$ if and only
if it has $n$ distinct real roots.\end{corollary}
We leave the following question open for further
investigations. Partial evidence for it has been given from the results in \cite{coot}, where it
has given a positive answer for $n\leq 5$, and where the reader can find references for the
existing literature on rank problems for real tensors.
\begin{question} Are all the ranks $\lfloor n/2\rfloor +1\leq k\leq n$ typical for real binary forms of degree $n$?\end{question}
\section{Main Theorem}
Let $f(x,y)$ be a real homogeneous polynomial of degree $n\geq 3$ without multiple roots in $\mathbb{C}$. Then $\nabla f(x,y)\not=(0,0)$ for any $(x,y)\not=(0,0)$ and one can define the maps $\bar{\phi}:S^1\to S^1$ and $\bar{\psi}:S^1\to S^1$ setting, for any $(x,y)$ with $x^2+y^2=1$, $\bar{\phi}(x,y)=|\nabla f|^{-1}(f_x,f_y)$ and
$\bar{\psi}(x,y)=|\nabla f|^{-1}(xf_x+yf_y, -y f_x+xf_y)$, with $|\nabla f|=(f_x^2+f_y^2)^{1/2}$. Setting
$(x,y)=(\cos\theta,\sin\theta)$, one can also write $\bar{\phi}$ and $\bar{\psi}$ as functions of $\theta$. \vskip1mm \noindent
{\bf Notation.} We denote $\partial_\theta=-y\partial_x+x\partial_y$ the basis tangent
vector to $S^1$ at the point $(x,y)$. Given any differentiable map $\phi:S^1\to M$ to a differentiable manifold $M$, we denote $\phi_\ast:T_\theta S^1\to T_{\phi(\theta)}M$ the associated tangent map. If $M=S^1$, and the map $\phi$ is defined in terms of angular coordinates by the function $\theta_1(\theta)$, we recall that the \emph{degree}, or \emph{winding number}, of $\phi$ is the number
\[\deg \phi=\frac{1}{2\pi}\int_0^{2\pi} \theta_1'(\theta)d\theta.\]
This is always an integer number, and for any $z\in S^1$ one has
$\# \phi^{-1}(z)\geq |\deg \phi\, |$. \vskip1mm \noindent 
The following lemmas are straightforward calculations and their proofs are omitted.
\begin{lemma}\label{lm:winding} Assume that $\theta_1'(\theta)$ never vanishes. Then $\# \phi^{-1}(z)=|\deg \phi\,|$ for any $z\in S^1$.\end{lemma}

We assume that for any $(\alpha ,\beta)\in\mathbb{P}^1(\mathbb{R})$ the
polynomial $\alpha f_x+\beta f_y$ has $n-1$ distinct roots in $\mathbb{R}$. Under this
assumption, we want to show that the absolute value of the degree of $\bar{\psi}$ is $n$. Since
$f(x,y)=0$ iff $\bar{\psi}(x,y)=(0,\pm 1)$ this implies that $f$ has all its roots in
$\mathbb{R}$. Indeed $\bar{\psi}(-x,-y)=(-1)^n\bar{\psi}(x,y)$, henceforth when $n$ is even
$n/2$ real roots of $f(x,y)=0$ are in $\bar{\psi}^{-1}(0,1)$ and the other $n/2$ roots are in
$\bar{\psi}^{-1}(0,-1)$; otherwise when $n$ is odd one gets $\bar{\psi}^{-1}(0,1)=
\bar{\psi}^{-1}(0,-1)$, hence $\bar{\psi}^{-1}(0,1)$ is the set of the $n$ real roots of
$f(x,y)=0$.
\begin{lemma}\label{lm:partialtheta} Let $F:S^1\to \mathbb{R}^2$ be a differentiable function defined by
$F(x,y)=(F_1(x,y),F_2(x,y))=(a(\theta),b(\theta))$. Then $F_\ast(\partial_\theta)=A\partial_x+B\partial_y$ with
$$\begin{array}{l}A=-y F_{1x}+xF_{1y}=a'(\theta)\\ B=-yF_{2x}+xF_{2y}=b'(\theta).\end{array}$$\end{lemma}
\noindent 
{\bf Notation.} Given a map $f:S^1\to \mathbb{R}^2$,
which one can write $f(\theta)=(a(\theta),b(\theta))$, we denote
with $(f,f_\theta)$ the matrix 
$$ \left(\begin{array}{ll}a(\theta)&
b(\theta)\\a'(\theta)& b'(\theta)\end{array}\right).$$
Notice that the sign of the determinant of this matrix expresses\\
if $f_\ast$ is orientation-preserving at the point $f(\theta)$.
\begin{lemma}\label{lm:normaliz}
Let $g:S^1\to \mathbb{R}^2$ and $\rho:S^1\to \mathbb{R}_+$ be differentiable functions. Then
$\det(g,g_\theta)=\rho^{-2}\det(\rho g,(\rho g)_\theta)$.\end{lemma}
Notice that if $\bar{g}:S^1\to S^1$ is the map $\bar{\phi}(x,y)=|\nabla f|^{-1}(f_x,f_y)$ then one may calculate
the sign of  $\det(\bar{\phi},\bar{\phi}_\theta)$ by reducing to the simpler map $\phi=(f_x,f_y):S^1\to \mathbb{R}^2$.
\vskip1mm\noindent {\bf Notation.} We denote by $H(f)=\det\left(\begin{array}{ll} f_{xx}& f_{xy}\\ f_{xy}&
f_{yy}\end{array}\right)$, the {\em hessian} of $f$.
\begin{proposition}\label{prop:hessian}Let $\phi:S^1\to \mathbb{R}^2$ be the map defined above. Then
$$\det(\phi,\phi_\theta)=(n-1)^{-1}H(f).$$ \end{proposition}
\begin{proof} We have $\phi_\ast(\partial_\theta)=A \partial_x+B\partial_y$ with $A$ and $B$ determined as in Lemma \ref{lm:partialtheta}, hence
\begin{eqnarray*} \det(\phi,\phi_\theta)&=&\det\left(\begin{array}{ll}f_x& f_y\\-yf_{xx}+xf_{xy}&-yf_{yx}+xf_{yy}\end{array}\right)\\
&=&\frac{1}{n-1}\det
\left(\begin{array}{cc}xf_{xx}+yf_{xy}& xf_{yx}+yf_{yy}\\-yf_{xx}+xf_{xy}&-yf_{yx}+xf_{yy}\end{array}\right)\\
&=&\frac{1}{n-1}\det\left(\begin{array}{cc}x&y\\
-y&x\end{array}\right)
\det\left(\begin{array}{cc} f_{xx}& f_{xy}\\ f_{xy}& f_{yy}\end{array}\right)\\
&=&\frac{1}{n-1} H(f)\end{eqnarray*}
\end{proof}
\begin{proposition}\label{prop:rotation} Let $\phi:S^1\to\mathbb{R}^2\cong\mathbb{C}$ defined by $\phi(\theta)=a(\theta)+ib(\theta)$ and $\psi:S^1\to\mathbb{C}$ defined by $\psi(\theta)=e^{-i\theta}\phi(\theta)$.  Then $\det(\psi,\psi_\theta)=\det(\phi,\phi_\theta)-a^2-b^2$.
\end{proposition}
\begin{proof} We calculate $\psi'(\theta)=(a'+b+i(b'-a))e^{-i\theta}$. It follows that
\[
\det(\psi,\psi_\theta)=\det\left(\begin{array}{ll}a&b\\a'+b&b'-a\end{array}\right)=
\det\left(\begin{array}{ll}a&b\\a'&b'\end{array}\right)-a^2-b^2.\]\qed\end{proof}
Notice that by Lemma \ref{lm:normaliz}, if $\phi:S^1\to \mathbb{R}^2$ is given in polar coordinates by
$\phi(\theta)=\rho(\theta)\bar{\phi}(\theta)$, with $\bar{\phi}:S^1\to S^1$, then
$\det(\phi,\phi_\theta)=\rho^2\det(\bar{\phi},\bar{\phi}_\theta)$. Moreover, expressing $\bar{\phi}$  in terms of angular coordinates by means of a function $\theta_1(\theta)$,
one sees easily that $\det(\bar{\phi},\bar{\phi}_\theta)=\theta_1'(\theta)$. We are
interested in $\bar{\phi}=|\nabla f|^{-1}(f_x,f_y)$ and
$\bar{\psi}=|\nabla f|^{-1}(xf_x+yf_y,-yf_x+xf_y)$. In this case
we get the following result.
\begin{corollary}\label{cor:theta} In the notations above, the following statements hold.\begin{enumerate}
\item
$\det(\bar{\phi},\bar{\phi}_\theta)=\theta_1'(\theta)=(n-1)^{-1}|\nabla
f|^{-2}H(f)$. \item $\deg \bar{\psi}=\deg
\bar{\phi}-1$.\end{enumerate}\end{corollary}
\begin{proof} The first statement follows form Proposition \ref{prop:hessian} and Lemma \ref{lm:normaliz}.
The second one follows from Proposition \ref{prop:rotation} and Lemma \ref{lm:normaliz} since\[
\deg\bar{\psi}=\frac{1}{2\pi}\int\det(\bar{\psi},\bar{\psi}_\theta)=\frac{1}{2\pi}\int\det(\bar{\phi},\bar{\phi}_\theta)-1
=\deg\bar{\phi}-1.\]\end{proof} Now we are ready to complete the proof of Theorem
\ref{thm:main}.
\begin{proof}[Proof of Theorem \ref{thm:main}]
Since for any
$(\alpha,\beta)\in\mathbb{R}^2\setminus\{0\}$ the polynomial $\alpha f_x+\beta f_y$ has $n-1$
distinct real roots, then the map $(f_x,f_y):\mathbb{P}^1_\mathbb{R}\to \mathbb{P}^1_\mathbb{R}$
has no ramification at any real point of $\mathbb{P}^1_\mathbb{R}$. Equivalently, the jacobian
of $\bar{\phi}$ which is equal to the hessian $H(f)$ is always non-zero at the real points of
$\mathbb{P}^1_\mathbb{R}$. We call the map $\bar{\phi}:S^1\to S^1$ defined by
$\bar{\phi}=|\nabla f|^{-1}(f_x,f_y)$ and we also express it as
$\theta_1=\theta_1(\theta)$ in angular coordinates. By the observation above 
 and Corollary \ref{cor:theta}
it follows that the derivative $\theta_1'(\theta)$ is non vanishing at any $\theta\in S^1$. Hence $\theta_1'(\theta)$ is either always positive or always negative.

\noindent {\em Claim:} $\theta_1'(\theta)<0$ for any $\theta$.

The sign of $\theta_1'(\theta)$ is the same than the sign of $H(f)$. Since we already know that it is constant it will be sufficient to evaluate it at a single point $(x,y)\in S^1$. We choose to examine the point $(1,0)$. One observes that for any binary form $g(x,y)={m\choose 0}a_0x^m+{m\choose 1}a_1 x^{m-1}y+\cdots+{m\choose m}a_my^m$ of degree $m\geq 3$, the Hessian $H(g)$ calculated at $(1,0)$ is equal to
\[m(m-1)\det\left(\begin{array}{ll}a_0&a_1\\a_1& a_2\end{array}\right).\] Similarly the Hessian of its derivative $g_x$ at $(1,0)$ is given by \[m(m-1)(m-2)\det\left(
\begin{array}{ll}a_0&a_1\\a_1&a_2\end{array}\right).\]
Therefore we find that $$H(g)(1,0)=(m-2)^{-1}H(g_x)(1,0).$$ Applying this result to $g=f$, we are reduced to compute the sign of $H(f_x)$. We know that $f_x$ has $n-1$ distinct real roots, so all of its
derivatives $\partial_x^i(f_x)$ have all their roots real and distinct, up to $i=n-3$.
The last of these derivatives is $h=\partial_x^{n-2}f$, and its Hessian is a constant equal to
$-\Delta(h)$,  hence $H(h)<0$. Applying recursively the reduction step, we find that
$H(f)(1,0)=(n-2)^{-1}H(f_x)(1,0)=((n-2)!)^{-1}H(h)<0,$ proving the claim.

By Corollary \ref{cor:theta}({\em 1.}), Lemma \ref{lm:winding} and applying the claim above, we
get that $\deg\bar{\phi}<0$ and $\#\phi^{-1}(z)=|\deg \phi\,|$ for any $z\in S^1$, hence
$\deg\bar{\phi}=-n+1$. Moreover, by Corollary \ref{cor:theta}({\em 2.}), we have
$\deg\bar{\psi}=\deg\bar{\phi} -1=-n$, hence $\#\mbox{real roots}(f)\geq |\deg\bar{\psi}\,|=n$.
This completes the proof of the Theorem.\end{proof}
We conclude giving a self-contained proof of the result on the rank of a real binary form
mentioned in the introduction. The arguments given are all already in \cite{coot}.
\begin{proof}[Proof of Corollary \ref{cor:rank}.]
 The statement holds for $n=3$, as shown in
\cite{coot}, Proposition 2.2. Assuming $n>3$, suppose the statement holds in degree $n-1$.
Assume rank$(f)=r$, so one can write $f=\lambda_1l_1^n+\cdots+\lambda_rl_r^n$, with $r$ minimal.
Then one can consider $l=l_1$ and $m=l_r$ and $g(t)=m^{-n}f$, with $t=l/m$. One sees that
$m^{-n+1}f_l=g'(t)$ can be expressed as a sum of at most $r-1$ $n$-th powers of linear forms in
$t$. If $f$ has $n$ distinct real roots then, by induction hypothesis $f_l$ has $n-1$ distinct
real roots and we find $r-1\geq n-1$, i.e. $r\geq n$. Since the inequality $r\leq n$ always
holds, as shown in \cite{coot} Proposition 2.1, we have $r=n$. Conversely, if the rank of $f$ is
$n$ then take $r=n$ and consider any derivative $\alpha f_x+\beta f_y=f_l$, after defining a
suitable coordinate system $l,m$, as explained in the introduction. We can consider the
polynomial $g'(t)=m^{-n+1}f_l$. If it has rank $<n-1$ then by indefinite integration over $t$ one
sees easily that $f$ has rank$<n$, contrary to the assumption. So rank$(f_l)=n-1$ and, by
induction hypothesis, it also holds that $f_l$ has $n-1$ distinct real roots. By the
arbitrariness of $l$ and by Theorem \ref{thm:main}, we conclude that $f$ has $n$ distinct
roots.
\end{proof}
\noindent
\textbf{Acknowledgements} The second author thanks G.~Ottaviani for posing the question giving rise to this paper and for stimulating discussions.

\newpage
\noindent
Antonio Causa\\
\texttt{causa@dmi.unict.it}

\noindent
Riccardo Re\\
\texttt{riccardo@dmi.unict.it}

\noindent
Universit\`a di Catania\\
Dipartimento di Matematica e Informatica,\\
Viale Andrea Doria 6\\
I--95125 Catania, Italy.

\end{document}